\documentclass[12pt]{article}
\usepackage{bbm}
\usepackage{amsmath}
\usepackage{hyperref}
\allowdisplaybreaks

\renewcommand{\emph}[1]{\textit{#1}}
\usepackage{enumerate,amsmath,amsthm,latexsym,amssymb}
\usepackage{color}\usepackage{graphicx}

\definecolor{brown}{cmyk}{0, 0.72, 1, 0.45}
\definecolor{grey}{gray}{0.5}
\newcommand{\red}[1]{#1}

\newcommand{\old}[1]{}

\newcounter{rot}

\newcommand{\card}[1]{\left|#1\right|}

\newcommand{\ignore}[1]{}

\newcommand{\set}[1]{\left\{#1\right\}}

\def\cP{\mathcal{P}}
\def\cQ{\mathcal{Q}}

\newcommand{\proofend}{\hspace*{\fill}\mbox{$\Box$}\\ \medskip\\ \medskip}

\def\ii_(#1,#2){i_{#1}^{#2}}

\def\LL{\Lambda}

\def\bx{{\bf x}}

\def\a{\alpha}
\def\b{\beta}
\def\d{\delta}

\def\e{\varepsilon}
\def\f{\phi}

\def\G{\Gamma}

\def\l{\lambda}

\def\n{\nu}
\def\p{\pi}
\def\P{\Pi}
\def\r{\rho}

\def\s{\sigma}

\def\up{\upsilon}

\def\1{{\bf 1}}
\def\0{{\bf 0}}

\newcommand{\rdup}[1]{\left\lceil #1 \right\rceil}

\def\cT{\mathcal{T}}

\newcommand{\brac}[1]{\left( #1 \right)}

\def\E{{\mathbb E}}

\def\Pr{\mathbb{P}}
\newcommand\bfrac[2]{\left(\frac{#1}{#2}\right)}

\def\bx{{\bf x}}

\def\2G{{\sc 2greedy}}

\newcommand{\nospace}[1]{}

\def\path{\operatorname{PATH}}

\newtheorem{theorem}{Theorem}[section]

\newtheorem{lemma}[theorem]{Lemma}
\newtheorem{corollary}[theorem]{Corollary}

\newtheorem{remthm}[theorem]{Remark}

\newcounter{thmtemp}

\usepackage[ruled, linesnumbered]{algorithm2e}
\def\gnm3{G_{n,m}^{\delta\geq 3}}
\def\Gnm3{{\mathcal G}_{n,m}^{\delta\geq 3}}
\newcommand{\beq}[2]{\begin{equation}\label{#1}#2\end{equation}}
\def\G{\Gamma}

\def\cH{{\mathcal H}}

\def\ec{5}

\def\vL{\vec{L}}
\def\vf{\vec{f}}
\def\vcT{\vec{\cT}}
\def\vT{\vec{T}}

\def\vcP{\vec{\cP}}
\def\vcQ{\vec{\cQ}}
\def\vK{\vec{K}}
\def\vX{\vec{X}}
\def\vP{\vec{P}}

\newcommand{\multstar}[1]{\begin{multline*}#1\end{multline*}}

\def\lo{\log n}

\def\LL{\Lambda}
\newcommand{\gk}[1]{\l^{#1}(1)}

\begin{document}
\title{A scaling limit for the length of the longest cycle in a sparse random digraph}
\author{ Michael Anastos\thanks{Department of Mathematics and Computer Science, Freie Universit{\"a}t Berlin, Berlin, Germany, email:{\bf manastos@zedat.fu-berlin.de}} and Alan Frieze\thanks{Department of Mathematical Sciences, Carnegie Mellon University, Pittsburgh PA, U.S.A. email:{\bf alan@random.math.cmu.edu}; the author is supported in part by NSF Grant DMS1363136} }

\maketitle

\begin{abstract}
We discuss the length $\vL_{c,n}$ of the longest directed cycle in the sparse random digraph $D_{n,p},p=c/n$, $c$ constant. We show that for large $c$ there exists a function $\vf(c)$ such that $\vL_{c,n}/n\to \vf(c)$ a.s. The function $\vf(c)=1-\sum_{k=1}^\infty p_k(c)e^{-kc}$ where $p_k$ is a polynomial in $c$. We are only able to explicitly give the values $p_1,p_2$, although we could in principle compute any $p_k$. 
\end{abstract}
\section{Introduction}
In this paper we consider the length $\vL_{c,n}$ of the longest cycle in the random digraph $D_{n,p},p=c/n$ where we will assume that $c$ is a sufficiently large constant. Here $D_{n,p}$ is the random subgraph of the complete digraph $\vK_n$ obtained by including each of the $n(n-1)$ edges independently with probability $p$. Most of the literature on long cycles has been concerned with the length  $L_{c,n}$ of the longest cycle in the random graph $G_{n,p}$. It was shown by Frieze \cite{Fpath} that w.h.p. $L_{c,n}\geq (1-(c+1+\e_c)e^{-c})n$ where $\e_c\to 0$ as $c\to\infty$. Using the elegant coupling argument of McDiarmid \cite{McD1} we see that this implies that w.h.p. $\vL_{c,n}\geq (1-(c+1+\e_c)e^{-c})n$. This was imp[roved by Krivelevich, Lubetzky and Sudakov \cite{KLS1} who showed that w.h.p. $\vL_{c,n}\geq (1-(2+\e_c)e^{-c})n$. Recently, Anastos and Frieze \cite{AF1} have shown that if $c$ is sufficiently large then w.h.p. $L_{c,n}\approx f(c)n$ as $n\to \infty$, for some function $f(c)$\footnote{ Here we say $A_n\approx B_n$ if $A_n/B_n\to 1$ as $n\to \infty$.}. 

In this paper we use the ideas of \cite{AF1} and show that w.h.p. $\vL_{c,n}\approx \vf(c)n$ and compute the first few terms of $\vf(c)=1-\sum_{k=1}^\infty p_k(c)e^{-kc}$ where $p_k(c)$ is a polynomial in $c$ for $k\geq 1$. I.e. we prove a scaling limit for $\vL_{c,n}$. The important point here is that we establish high probability errors that tend to zero with $n$, regardless of $c$.

Let $K_1$ denote the giant strong component of $D_{n,p}$, as discovered by Karp \cite{K}. We consider a process that builds a large Hamiltonian subgraph of $K_1$. Our aim is to construct (something close) to a copy of the random graph $D_{5-in,5-out}$ as a large subgraph of $K_1$. In the random graph $D_{k-in,k-out}$ each $v\in [n]$ independently chooses $k$ in-neighbors and $k$  outneighbors to make a digraph with $\approx 2kn$ random edges. It has been shown by Cooper and Frieze \cite{CF1}, {\cite{CF1a} that $D_{k-in,k-out}$ is Hamiltonian w.h.p. provided that $k\geq 2$. Taking $k=5$ as opposed to $k=2$ will greatly simplify the discussion. In order to do this, we will construct $D_{n,p}$ as the union of two independent copies $D_{red},D_{blue}$ of $D_{n,q}$ where $1-p=(1-q)^2$ so that $q=\frac{c}{2n}+O(n^{-2})$. One copy will have red edges and the other copy will have blue edges. A red edge $(v,w)$ will be assoiciated with the vertex $v$ and a blue edge $(v,w)$ will be associated with the vertex $w$. In this way, the vertex $v$ will be incident to a random number of red out-edges and to a random number of blue in-edges. These edge sets will be independent by construction. We say in the following that $w$ is a blue in-neighbor of $v$ if $(w,v)$ is an edge 
of $D_{blue}$ and that $w$ is a red out-neighbor of $v$ if $(v,w)$ is an edge of $D_{red}$.

We construct a sequence of sets $S_0=\emptyset,S_1,S_2,\ldots,S_L\subseteq K_1$ as follows:  suppose now that we have constructed $S_\ell$, $\ell \geq 0$. We construct $S_{\ell+1}$ from $S_{\ell}$ via one of two cases: 

{\bf Construction of $S_L$}\\
{\bf Case a:} If there is a vertex $v\in S_\ell$ that has at most four blue in-neighbors outside $S_\ell$ then we add the blue in-neighbors of $v$ outside $S_\ell$ to $S_\ell$ to make $S_{\ell+1}$. Similarly, if there is a vertex $v\in S_\ell$ that has at most four red out-neighbors outside $S_\ell$ then we add the red out-neighbors of $v$ outside $S_\ell$ to $S_\ell$ to make $S_{\ell+1}$.
\\
{\bf Case b:} If there is a vertex $v\in K_1\setminus S_\ell$ that has at most four blue in-neighbors in $K_1\setminus S_\ell$ then we add $v$ and the blue in-neighbors of $v$  to $S_\ell$ to make $S_{\ell+1}$. Similarly, if there is a vertex $v\in K_1\setminus S_\ell$ that has at most four red out-neighbors in $K_1\setminus S_\ell$ then we then we add $v$ and the red out-neighbors of $v$  to $S_\ell$ to make $S_{\ell+1}$.

$S_L$ is the set we end up with when there are no more vertices to add. We note that $S_L$ is well-defined and does not depend on the order of adding vertices. Indeed, suppose we have two distinct outcomes $O_1= v_1,v_2,\ldots,v_r$ and $O_2= w_1,w_2.,\ldots,w_s$. Assume without loss of generality that there exists $i$ which is the smallest index such that $w_i\notin O_1$. Then, $X=\set{w_1,w_2,\ldots,w_{i-1}}\subseteq O_1=\set{v_1,v_2,\ldots,v_r}$.
If $w_i$ invoked Case a or Case b then $w_i$ has at most 4 blue in-neighbors or at most 4 red out- neighbors in $K_1 \setminus X$ hence in 
$K_1 \setminus O_1 \subseteq K_1 \setminus X$. This contradicts the fact that $w_i\notin O_1$. Otherwise $w_i$ was added to $X$ because there exists a vertex $u \in X$ such that $w_i$ is a blue in-neighbor (or a red out-neighbor respectively) of $u$ and $u$ has at most 4 blue in-neigbors (red out-neighbors resp.) in $K_1\setminus X$. Thus $u \in O_1$ has at most 4 blue in-neigbors (red out-neighbors resp.) in $K_1\setminus X \subseteq K_1\setminus X$. Once again, this contradicts the fact that $w_i\notin O_1$.


We will argue below in Section \ref{sizes} that w.h.p. the graph $\G_L$ {\em underlyng} the digraph $D_L$ induced by $S_L$ is a forest plus a few small components (the graph underlying a digraph is obtained by ignoring orientation). Each tree in $\G_L$ will w.h.p. have at most $\log n$ vertices and w.h.p. $\G_L$ will have $o(n)$ vertices lying on non-tree components. From now on, when we refer to trees, they are either trees of $\G_L$ or digraphs whose underlying graphs are trees of $\G_L$.

{\bf Notation 1:} Let $\vcT$ denote the set of trees in $\G_L$. Each tree $T$ of $\G_L$ will appear as a digraph $\vT$ in $D_L$ when we take account of orientation. For $\vT\in \vcT$ let $\vcP_T$ be the set of vertex disjoint packings of {\em properly oriented} paths in $\vT$ where we allow only paths whose start vertex has in-neghbors $K_1 \setminus V(\vT)$ and whose end vertex has out-neighbors in $K_1 \setminus V(\vT)$. Here we allow paths of length 0, so that a single vertex with neighbors in $K_1 \setminus V(\vT)$ counts as a path. For $P\in \vcP_T$ let $n(\vT,P)$ be the number of vertices in $\vT$ that are not covered by $P$. Let $\f(\vT)=\min_{P\in \vcP_T} n(\vT,P)$ and $\vcQ(\vT)\in \vcP$ denote a set of paths that leaves $\f(\vT)$ vertices of $\vT$ uncovered i.e. satisfies $n(\vT,\vcQ(\vT))=\f(\vT)$.

We will prove
\begin{theorem}\label{th1}
Let $p=c/n$ where $c>1$ is a sufficiently large constant. Then w.h.p.
\beq{sizeC}{
\vL_{c,n}\approx |V(K_1)|-\sum_{\vT\in \vcT}\f(\vT).
}
\end{theorem}
The RHS of \eqref{sizeC}, modulo the $o(n)$ vertices that are spanned by non-tree components in $\G_L$, is clearly an upper bound on the largest directed cycle in $K_1$. Any cycle must omit at least $\f(\vT)$ vertices from each $\vT\in\vcT$. On the other hand, as we show below, w.h.p. there is cycle $H$ that spans $V^*=(K_1\setminus S_L)\cup \bigcup_{T\in \cT}V(\cQ(T))$. The length of $H$ is equal to the RHS of \eqref{sizeC}.

The size of $K_1$ is well-known. Let $x$ be the unique solution of $xe^{-x}=ce^{-c}$ in $(0,1)$. Then w.h.p. (see e.g. \cite{FK}, Theorem 13.2),
\begin{align}
|K_1|&\approx \brac{1-\frac{x}{c}}^2n.\label{C2v}
\end{align}
Equation (4.5) of Erd\H{o}s and R\'enyi \cite{ER} tells us that 
\beq{xval}{
x=\sum_{k=1}^\infty\frac{k^{k-1}}{k!}(ce^{-c})^k=ce^{-c}+c^2e^{-2c}+O(c^3e^{-3c}).
}
We will argue below that w.h.p., as $c$ grows, that
\beq{small}{
\sum_{\vT\in \vcT}\f(\vT)=(c^2e^{-2c}+O(c^3e^{-3c}))n.
}
The term $c^2e^{-2c}n$ arises from vertices of out-degree one sharing a common out-neighbor or vertices of in-degree one sharing a common in-neighbor.

We therefore have the following improvement to the estimate in \cite{KLS1}.
\begin{corollary}\label{cor1} 
W.h.p., as $c$ grows, that
\beq{sizeC1}{
\vL_{c,n}\approx \brac{1-2e^{-c}-(c^2+2c)e^{-2c}-O(c^3e^{-3c})}n.
}
\end{corollary}
Note the term and  accounts for vertices of in- or out-degree 0. In principle we can compute more terms than what is given in \eqref{sizeC1}. We claim next that there exists some function $\vf(c)$ such that the sum in \eqref{sizeC} is concentrated around $\vf(c)n$. In other words, the sum in \eqref{sizeC} has the form  $\approx \vf(c)n$ w.h.p.
\begin{theorem}\label{limit}
\begin{enumerate}[(a)]
\item There exists a function $\vf(c)$ such that for any fixed $\epsilon>0$, there exists $n_\e$ such that for $n\geq n_\e$,
\beq{eq:expectation}{
\card{\frac{\E[\vL_{c,n}]}{n}-\vf(c)}\leq \epsilon.
}
\item
\[
\frac{\vL_{c,n}}{n}\to \vf(c)\ a.s.
\]
\end{enumerate}
\end{theorem}
We will prove Theorem \ref{limit} in Section \ref{seclimit}.
\subsection{Structure of $D_L$:}\label{sizes}
We first bound the size of $S_L$. We need the following lemma on the density of small sets.
\begin{lemma}\label{lem1}
W.h.p., every set $S\subseteq [n]$ of size at most $n_0=n/10c^3$ contains less than $3|S|/2$ edges in $D_{n,p}$.
\end{lemma}
\begin{proof}
The expected number of sets invalidating the claim can be bounded by
\multstar{
\sum_{s=4}^{n_0}\binom{n}{s}\binom{s(s-1)}{3s/2}\bfrac{c}{n}^{3s/2}\leq \sum_{s=4}^{n_0}\brac{\frac{ne}{s}\cdot\bfrac{2se}{3}^{3/2}\cdot\bfrac{c}{n}^{3/2}}^s\\
=\sum_{s=4}^{n_0} \bfrac{e^{5/2}(2c)^{3/2}s^{1/2}}{3^{3/2}n^{1/2}}^s=o(1).
}
\end{proof}
Now consider the construction of $S_L$. Let $A$ be the set of the vertices with blue in-degree less than $D=100$ or red out-degree less than $D$. Let $S_0'=(A \cup N_b(A) \cup N_r(A))\cap S_L\subseteq S_L$, where $N_b(A)$ is the set of blue in-neighbors of vertices in $A$ and $N_r(A)$ is the set of red out-neighbors of vertices in $A$. If we start with $S_0=S_0'$ and run the process for constructing $\Gamma_L$ then we will produce the same $S_L$ as if we had started with $S_0= \emptyset$. This is because, as we have shown, the order of adding vertices does not matter. Now w.h.p. there are at most $n_D=\frac{2c^De^{-c}}{D!}n$ vertices of blue in-degree at most $D$ or red out-degree $D_{n,p}$, (see for example Theorem 3.3 of \cite{FK} that deals with the same question as it relates to degrees in $G_{n,p}$). It follows that w.h.p. $|S_0'|\leq ne^{-2c/3}$. 

Now suppose that the process runs for another $k$ rounds. Then $S_{k}$ contains at least $kD$ edges and at most $Dn_D+5k$ vertices. This is because round $k$ adds at most five {\em new} vertices to $S_k$ and the $k$ vertices that take the role of $v$ have either (i) blue in-degree at least $D$ with all blue in-neighbors in $S_k$ or (i) red out-degree at least $D$ with all red out-neighbors in $S_k$. If $k$ reaches $2n_D$ then 
\[
\frac{e(S_k)}{|S_k|}\geq \frac{2Dn_d}{(D+10)n_d}>\frac32.
\]
So, by Lemma \ref{lem1}, we can assert that w.h.p. the process runs for less than $2n_D$ rounds and,
\beq{smallS}{
|V(\G_L)|\leq (D+10)n_D\leq ne^{-c/2}.
}
We note the following properties of $S_L$. Let 
$$
V_1=K_1\setminus S_L\text{ and }V_2=\{v\in  S_L:\;v \text{{  has at least one blue in-neighbor and at least one red }}
$$ 
$$\text{ out-neighbor  in } V_1\} .
$$
Then, 
\begin{enumerate}[{\bf G1}]
\item Each vertex $v\in S_L \setminus V_2$ has no blue in-neighbors or no red out-neighbors  in $V_1$.
\item Each $v\in V_1 \cup V_2$ has at least five blue in-neighbors and five red out-neighbors in $V_1$.
\end{enumerate}
Now consider a component $K$ of $\Gamma_L$. Let $C_0=C_0(K)=\set{v_1,v_2,\ldots,v_L}$ denote the set of vertices in $K$ that are $v$ in some step in the construction of $D_L$, indexed by the round in which they are added. We will prove by induction on $i$ that for $0\leq i \leq L$ and each component $K$ spanned by $S_i$, 
\beq{vK}{
|C_{0}(K)|\geq  \frac{|K|}{5}.
}
$S_0=\emptyset$ and so for $i=0$, \eqref{vK} is satisfied by every component spanned by $S_0$. Suppose that at step $i$,  \eqref{vK} is satisfied by every component spanned by $S_i$.

At step $i+1$, $v_{i+1}$ invokes either Case a or Case b. In both cases  $S_{i+1}=S_i \cup \big(\{v_{i+1}\} \cup N_c(v_{i+1})\big)$, where $c\in\set{b,r}$. The addition of the new vertices into $S_i$  could merge components $K_1,K_2,\ldots,K_r$  into one component $K'$. We add $v_{i+1}$ plus at most four other vertices to $K'$, hence $|K'|\leq \sum_{j\in [r]} |K_i|+5$. In addition every vertex that contributed to $C_0(K_j),\,j=1,2,...,r$ now contributes towards $C_0(K')$. The inductive hypothesis implies that $|C_0(K_j)| \geq |K_j|/5$ for $j\in [r]$. Thus,
\[
|C_0(K')|\geq 1+\sum_{j \in [r]}|C_0(K_j)| \geq 1+\frac{1}{5}\sum_{j \in [r]}|K_j|\geq 1+\frac{|K'|-5}{5} = \frac{|K'|}{5}.
\]
and so \eqref{vK} continues to hold for all the components spanned by $S_{i+1}$.

We next show that w.h.p., only a small component $K$ can satisfy \eqref{vK}. $K$ will have at least $|K|/5$ vertices for which either there are no blue in-neghbors outside $K$ or no red out-neighbors outside of $K$. It will also contain a spanning tree in the graph undelying $D_{n,p}$. So, the expected number of components of size $k\leq ne^{-c/2}$ that satisfy this condition is at most
\begin{align}\label{11}
\binom{n}{k}k^{k-2}\bfrac{c}{n}^{k-1}\binom{k}{k/ \ec}\times   \brac{2\brac{1-\frac{c}{2n}}^{(n-k)}}^{k/\ec}&
\leq \bfrac{ne}{k}^kk^{k-2}\bfrac{c}{n}^{k-1}2^{6k/5}e^{- ck/11} \nonumber
\\&\leq \frac{n}{ck^2}\brac{2^{6/5}ce^{1-c/11}}^k=o(n^{-2}),
\end{align}
if $c$ is large and $k\geq \log n$.

So, we can assume that all components are of size at most $\log n$. Then the expected number of vertices on components that are not trees is bounded by
\begin{align*}
\sum_{k=5}^{\log n}\binom{n}{k}k^{k+1}\bfrac{c}{n}^{k}\binom{k}{k/\ec} \times   \brac{2\brac{1-\frac{c}{2n}}^{(n-k)}}^{k/\ec}&
\leq \sum_{k=5}^{\log n}\bfrac{ne}{k}^kk^{k+1}\bfrac{c}{n}^{k}2^{6k/5}e^{-ck/11}\\
&\leq \sum_{k=5}^{\log n}k\brac{2^{6/5}ce^{1-c/11}}^k=O(1).
\end{align*}

The Markov inequality implies that w.h.p. such components span at most $\log n=o(n)$ vertices.
\section{Proof of Theorem \ref{th1}}
For $\vT\in\vcT$,  let $\vX_T$ be the set obtained by contracting each path $\vP$ of $\vcQ(\vT)$ to a vertex $v_{\vP}$ with blue in-neighbors in $V_1$ equal to the blue in-neighbors in $V_1$ of the start vertex of $\vP$ and red out-neighbors in $V_1$ equal to the red out-neighbors in $V_1$ of the end vertex of $\vP$. Note that the colors of the internal edges of a path $\vP$ do not play a role here. Let $\vX^*=\bigcup_{\vT\in \vcT}\vX_T$. By construction, the digraph induced by $V_1$ contains a copy of $D_{5-in,5-out}$ with $N=|V_1|$ vertices. Indeed, the blue edges contributing the 5-in edges and the red edges contributing the 5-out edges. For each $v\in V_1$, the blue in-neighbors form a random set of size at least five, independent of the other vertices in $V_1$. Similarly for the red out-neighbors. 

We let $D^*$ be the digraph with vertex set $V_1^*=V_1\cup \vX^*$ and a copy of $D_{5-in,5-out}$ on $V_1$ and for each $x\in \vX^*$ five red edges joining $x$ to $ V_1$ and five blue edges from $V_1$ to $x$.

Our next task is to prove that the random digraph $D^*$ defined in the previous section contains a Hamilton cycle. Let $H$ denote such a cycle through $V_1^*$. We obtain a Hamilton cycle of $V^*$ (defined following Theorem \ref{th1}) by uncontracting each path $\vP$ of $\vcQ(\vT)$. This will complete the proof of Theorem \ref{th1}. Our proof of the existence of $H$ will be very similar to the proof in Cooper and Frieze \cite{CF1}. It doesn't really offer any new technical insights and so we have placed the proof into an appendix.

\section{Proof of Theorem \ref{limit}}\label{seclimit}
For $\vT\in \vcT$ we let $v_0(\vT)$ denote the set of vertices in $\vT$ that do not have neighbors outside $\vT$. For $v\in K_1$ we let $\phi(v)= \phi(\vT)/|v_0(\vT)|$ if $v\in \up_0(T)$ for some $\vT \in \vcT$ and $\phi(v)=0$ otherwise. Thus
$$\sum_{T \in \vcT}\phi(\vT)=\sum_{v \in K_1} \phi(v).$$

Hence \eqref{sizeC} can be rewritten as,
\begin{equation}\label{sizes2}
\vL_{c,n} \approx |K_1|- \sum_{v \in K_1} \phi(v).
\end{equation}

Let $k_1=k_1(\epsilon,c)$ be the smallest positive integer such that
\[
\sum_{k=k_1-1}^\infty  (e^{\ec} 2^{6} ce^{-c/5})^k < \frac{\epsilon}{3}.
\]
Note that for large $c$, we have
\beq{sizek1}{
k_1\leq \frac{5}{c}\log\frac{1}{\e}.
}
To begin let $\vK_{\ec,\ec}$ denote the complete bipartite digraph with six vertices, five in each part of the partition. For $v\in K_1$ let $D_v$ be the digraph consisting of (i) the vertices of $D=D_{n,p}=D_{blue}\cup D_{red}$ that are within distance $k_1$ from $v$  and (ii) a copy of $\vK_{\ec, \ec}$ where every vertex in the $k_1$ neighborhood of $v$ is adjacent to each vertex of the same one part of the bipartition. Distance here is graph distance in the undirected graph underlying $D$. We consider the algorithm for the construction of $\Gamma_L$ on $G_v$ and let $K_{1,v},\Gamma_{L,v}, V_{1,v}, S_{L,v},\up_{0,v}(\vT)$ be the corresponding sets/quantities.

For a tree $\vT \in S_{L,v}$ let $\vf(\vT)$ be equal to $|\vT|$ minus the maximum number of vertices that can be covered by a set of vertex disjoint paths with endpoints in $V_{2,v}$ (we allow paths of length 0). For $v\in K_1$, if $v$ belongs to some tree $\vT \in S_{L,v}$ set $\vf(v)=\vf(\vT)/\up_{0,v}(\vT)$, otherwise set $\vf(v)=0$. 

For $v\in K_1$ let $t(v)=1$ if $v\in V_1$ or if $v\in S_L$ and in $\Gamma_L$, $v$ lies in a component with at most $k_1-2$ vertices in $ \G_L $. Set $t(v)=0$ otherwise. Observe that if $t(v)=1$ then $\phi(v)=\vf(v)$. Otherwise $|\phi(v)-\vf(v)| \leq 1$. 

By repeating the arguments used to prove \eqref{11} and \eqref{vK} it follows that if 
$t(v)=0$ then $v$ lies on a subgraph spanned by some set of vertices $K$ of size at most $\log n$. In addition at least $(|K|-1)/\ec$ vertices in $K\setminus{\{v}\}$ either do not have blue in-neighbors or red out-neighbors outside $K$. Thus the expected number of vertices $v$ satisfying $t(v)=0$ is bounded by
\begin{align*}
&\sum_{k=k_1-1}^{\log n} \sum_{j=k}^{\ec k}  \binom{n}{j} \binom{j}{k} j^{j-2}  (2p)^{j-1} \times  \brac{2\brac{1-\frac{p}{2}}^{(n-j)}}^k \\
& \leq  2n \sum_{k=k_1-1}^{\log^2 n}  \ec k \bfrac{e}{{\ec} k}^{\ec k} 2^{6k}(\ec k)^{{\ec} k-2} (2c)^{k-1}e^{-ck/5}\\ 
&\leq 2n  \sum_{k=k_1-1}^\infty (e^{\ec} 2^{7} ce^{-c/5})^k< \frac{\epsilon n}{3}.
\end{align*}
A vertex $v\in [n]$ is {\em good} if the $i$th level of its Breadth First Search (BFS) neighborhood has size at most $3 (2c)^i k_1/\epsilon$ for every $i\leq k_1$ and it is {\em bad} otherwise. Here the BFS is done on the graph underlying $D$. Because the expected size of the $i^{th}$ neighborhood is $\approx (2c)^i$ we have by the Markov inequality that $v$ is bad with probability at most $\approx \e/3k_1$ and so the expected number of bad vertices is bounded by $\e n/2$.  Thus
\begin{align*} 
\E\brac{\card{\sum_{v\in V}\phi(v)- \sum_{v\text{ is good }} \vf(v) }}
& \leq \E\brac{\card{\sum_{v\in V}\phi(v)- \sum_{v \in V} \vf(v)}}+ \E\brac{\card{\sum_{ v\text{ is bad }} \vf(v)}}\\
& \leq  \E\brac{\card{\sum_{v: t(v)=0}|\phi(v)-  \vf(v)}}+\E\brac{\sum_{v\text{ is bad }} 1}\\
& \leq \E\brac{\sum_{v: t(v)=0} 1} +\frac{\epsilon n}2\\
& \leq \frac{\epsilon n}3 +\frac{\epsilon n}2 < \epsilon n.
\end{align*}
Let $\mathcal{H}_\e$ be the set of BFS neighborhoods that are good i.e. whose $i$th levels are of size at most $3(2c)^i k_1/\epsilon$ for every $i\leq k_1$. Every element of $\mathcal{H}_\e$ corresponds to a pair $(H,o_H)$ where $H$ is a digraph and $o$ is a distinguished vertex of $H$, that is considered to be the root.  Also for $v\in K_1$ let $D(N_{k_1}(v))$ be the  subdigraph induced by the ${k_1}^{th}$ neighborhood of $v$. For $(H,o_H)\in \mathcal{H}_\e$ let $int(H)$ be the set of vertices incident to the first $k_1-1$ neighborhoods of $o_H$ and let $Aut(H,o_H)$ be the number of automorphisms of $H$ that fix $o_H$.  Note that each good vertex $v$ is associated with a pair $(H,o_H)\in\cH_\e$ from which we can compute $\vf(v)$, since $\vf(v)=\vf(o_H)$. Thus, if now 
\[
M=|E(K_1)|,N=|K_1|,
\]
\begin{align}
\E\brac{\sum_{v\text{ is good}}\vf(v)\bigg| M,N}&
=\sum_{v}\sum_{k\geq 1}\sum_{\substack{(H,o_H) \in \mathcal{H}_\e\\ (D(N_{k_1}(v)),v)=(H,o_H)\\|V(H)|=k}} \r_{H,o_H} \vf(o_H)\nonumber
\\&=o(n)+\sum_{v}\sum_{k\geq 1}\sum_{\substack{(H,o_H) \in \mathcal{H}_\e\\H\text{ is a tree}\\(D(N_{k_1}(v)),v)=(H,o_H) \\ |V(H)|=k}} \r_{H,o_H} \vf(o_H),\label{goodvert}
\end{align}
where $\r_{H,\s_H}$ is the probability  $(D(N_{k_1}(v)),v)=(H,o_H)$ in $K_1$. We show in Section \ref{model} that
\beq{piH}{
\r_{H,o_H}\approx  \frac{1}{Aut(H,o_H)} \bfrac{N}{M}^{k-1} \lambda^{2k-2} \frac{e^{2k\l}}{f_1(\l)^{2k}},
}
where $f_k$ is defined in \eqref{fk} below and $\l$ satisfies \eqref{2} below.

Finally observe that with the exception of the $o(1)$ term, all the terms in \eqref{goodvert} are independent of $n$. We let
\beq{hc}{
\vf_\e(c)= \sum_{k\geq 1}\sum_{\substack{(H,o_H) \in \mathcal{H}_\e\\H\text{ is a tree}}}  \frac{\vf(o_H)}{Aut(H,o_H)}\bfrac{N}{M}^{k-1} \lambda^{2k-2} \frac{e^{2k\l}}{f_1(\l)^{2k}}.
}
Then for a fixed $c$, we see that $\vf_\e(c)$ is monotone increasing as $\e\to 0$. This is simply because $\cH_\e$ grows. Furthermore, $\vf_\e(c)\leq 1$ and so the limit $\vf(c)=\lim_{\e\to0}f_\e(c)$ exists. This verifies part (a) of Theorem \ref{limit}. For part (b), we prove, (see \eqref{Azuma}),
\begin{lemma}
\[
\Pr(|\vL_{c,n}-\E(\vL_{c,n})| \geq \e n+n^{3/4}) = O(e^{-\Omega(n^{1/5})}).
\]
\end{lemma}
\begin{proof}
To prove this we show that if $\n(H)$ is the number of copies of $H$ in $K_1$ then $H\in \cH_\e$ implies that
\beq{nH}{
\Pr(|\n(H)-\E(\n(H))| \geq n^{3/5}) =O(e^{-\Omega(n^{1/5})}).
}
The inequality follows from a version of Azuma's inequality  (see \eqref{Azuma}), and the lemma follows from taking a union bound over 
\multstar{
\exp\set{O\bfrac{c^{k_1(\epsilon)}k_1(\epsilon)}{\epsilon}}=\exp\set{O\bfrac{c^{\frac{5\log \frac{1}{\e}}{c}}\frac{5\log \frac{1}{\e}}{c}}{\e}}\\
=\exp\set{O\bfrac{(1/\e)^{5\log c/c} \log \frac{1}{\e}}{c\e}}=\exp\set{O((1/\e)^{5+5\log c/c})}
}
graphs $H$.
 Note also that the $o(n)$ term in \eqref{goodvert} is bounded by the same $e^{O((1/\e)^{5+5\log c/c})}$ term times the number of cycles of length at most $2k_1$ in $G$. The probability that this exceeds $n^{1/2}$ is certainly at most the RHS of \eqref{nH}. We will give details of our use of the Azuma inequality in Section \ref{model}.
\end{proof}
Part (b) of Theorem \ref{limit} follows by letting $\e\to 0$ and from the Borel-Cantelli lemma.
\subsection{A Model of $K_1$}\label{model}
$K_1$ induces a random digraph with minimum in-degree and out-degree at least one. $K_1$ is distributed as a random strongly connected digraph with $N$ vertices and $M$ edges. This follows from the fact that each such digraph has the same number of extensions to a digraph with $n$ vertices and $m$ edges where $K_1$ is the unique giant strongly connected component. Most vertices of $K_1$ will have in-degree and out-degree close to $c$, since $c$ is large. It follows from Theorem 3 of Cooper and Frieze \cite{CF2} that a random digraph with this degree sequence has probability asymptotic to $e^{-\b}$ where $\b=\b(c)\to 0$ as $c\to\infty$. It follows from this that we can model the digraph induced by $K_1$ as a random digraph with $N$ vertices and $M$ edges. The probability of any event will be inflated by at most $(1+o(1))e^{\b}$ by conditioning on strong connecttvity. We denote this model by $D_{N,M}^{\pm1}$.
\subsubsection{Random Sequence Model}\label{refined}
This is essentially a repeat of Section 3.1.1 of \cite{AF1}. The differences are minor, but we feel we need to include the argument. We must now take some time to explain the model we use for $D_{N,M}^{\pm1}$. We use a variation on the pseudo-graph model of Bollob\'as and Frieze \cite{BollFr} and Chv\'atal \cite{Ch}. Given a sequence $\bx = (x_1,x_2,\ldots,x_{2M})\in [n]^{2M}$ of $2M$ integers between 1 and $N$ we can define a (multi)-digraph
$D_{\bx}=D_\bx(N,M)$ with vertex set $[N]$ and edge set $\{(x_{2i-1},x_{2i}):1\leq i\leq M\}$. The in-degree $d_{\bx,-}(v)$ of $v\in [N]$ and the out-degree $d_{\bx,+}(v)$ of $v\in [N]$ are given by 
$$d_{\bx,-}(v)=|\set{j\in [M]:x_{2j}=v}|\text{ and }d_{\bx,+}(v)=|\set{j\in [M]:x_{2j-1}=v}|.$$
If $\bx$ is chosen randomly from $[N]^{2M}$ then $D_{\bx}$ is close in distribution to $D_{N,M}$. Indeed,
conditional on being simple, $D_{\bx}$ is distributed as $D_{N,M}$. To see this, note that if $D_{\bx}$ is simple then it has vertex set $[N]$ and $M$ edges. Also, there are $M!$ distinct equally likely values of $\bx$ which yield the same digraph. 

Our situation is complicated by there being a lower bound of one on the minimum in-degree and out-degree. So we let
$$[N]^{2M}_{\delta\pm\geq 1}=\{\bx\in [N]^{2M}:d_{\bx,\pm}(j)\geq 1\text{ for }j\in[N]\}.$$
Let $D_\bx$ be the multi-graph $D_\bx$ for $\bx$ chosen uniformly from $[N]^{2M}_{\delta\pm\geq1}$. It is clear then that conditional on being simple, $D_\bx$ has the same distribution as $D_{N,M}^{\pm1}$. It is important therefore to estimate the probability that this graph is simple. For this and other reasons, we need to have an understanding of the degree sequence $d_\bx$ when $\bx$ is drawn uniformly from $[N]^{2M}_{\delta\pm\geq1}$. Let 
\beq{fk}{
f_k(\l)=e^\l-\sum_{i=0}^{k-1}\frac{\l^i}{i!}
}
for $k\geq 0$.
\begin{lemma}
\label{lem3}
Let $\bx$ be chosen randomly from $[N]^{2M}_{\delta\pm\geq1}$. Let $Y_j,Z_j,j=1,2,\ldots,N$ be independent copies of a {\em truncated Poisson} random variable $\cP$, where
$$\Pr(\cP=t)=\frac{{\l}^t}{t!f_1({\l})},\hspace{1in}t\geq 1.$$
Here ${\l}$ satisfies
\begin{equation}\label{2}
\frac{\l e^\l}{f_1({\l})}=\frac{M}{N}.
\end{equation}
Then $\{d_{\bx,-}(j)\}_{j\in [N]}$ is distributed as $\{Y_j\}_{j\in [N]}$ conditional on $Y=\sum_{j\in [n]}Y_j=M$ and
$\{d_{\bx,+}(j)\}_{j\in [N]}$ is distributed as $\{Z_j\}_{j\in [N]}$ conditional on $Z=\sum_{j\in [n]}Z_j=M$.
\end{lemma}
\begin{proof}
This can be derived as in  Lemma 4 of \cite{AFP}.
\end{proof}
We note that w.h.p.
\beq{MNsize}{
N\geq n(1-2e^{-c/2})\text{ and }M\in (1\pm\e_1)cN,
}
where $\e_1=c^{-1/3}$. The bound on $N$ follows from \eqref{C2v} and \eqref{smallS} and the bound on $M$ follows from the fact that in $G_{n,p}$,
\[
\Pr\brac{\exists S:|S|=N,e(S)\notin(1\pm\e_1)N(N-1)p}\leq 2\binom{n}{N}\exp\set{-\frac{\e_1^2N(N-1)p}{3}}=o(1).
\]
It follows from \eqref{2} and \eqref{MNsize} and the fact that $e^\l/f_1({\l})\to 1$ as $c\to \infty$ that for large $c$,
\beq{uplam}{
\l=c\brac{1+O(e^{-c})}.
}
We note that the variance $\s^2$ of $\cP$ is given by
\[
\s^2=\frac{\l(\l+1)e^\l f_1(\l)-\l^2e^{2\l}}{f_1^2(\l)}.
\]
Furthermore,
\begin{align}
\Pr\left(\sum_{j=1}^NY_j=M\right)&=\frac{1}{\s\sqrt{2\p N}}(1+O(N^{-1}\s^{-2}))\label{local1}\\
\noalign{and}
\Pr\left(\sum_{j=2}^NY_j=M-d\right)&=\frac{1}{\s\sqrt{2\p N}}\left(1+O((d^2+1)N^{-1}\s^{-2})\right). \label{local2}
\end{align}
This is an example of a local central limit theorem. See for example, (5) of \cite{AFP}. It follows by repeated application of \eqref{local1} and \eqref{local2} that if $k=O(1)$ and $d_1^2+\cdots+d_k^2=o(N)$ then
\beq{local3}{
\Pr\brac{Y_i=d_i,i=1,2,\ldots,k\mid\sum_{j=1}^NY_j=M}\approx \prod_{i=1}^k\frac{\l^{d_i}}{d_i!f_1(\l)}.
}
Let $\n_{\bx,-}(s)$ denote the number of vertices of in-degree $s$ in $D_\bx$ and let $\n_{\bx,+}(s)$ denote the number of vertices of out-degree $s$ in $D_\bx$. 
\begin{lemma}
\label{lem4x}
Suppose that $\log N=O((N {\l})^{1/2})$. Let $\bx$ be chosen randomly from $[N]^{2M}_{\delta\geq2}$. Then as in equation (7) of \cite{AFP}, we have that with probability $1-o(N^{-10})$,
\begin{align}
\left|\n_{\bx,\pm}(j)-\frac{N{\l}^j}{j!f_{1}({\l})}\right|& \leq \brac{1+\bfrac{N {\l}^j}{j!f_1({\l})}^{1/2}}\log^2 N,\ 1\leq j\leq \log N.\label{degconc}\\
\n_\bx(j)&=0,\quad j\geq \log N.\label{degconc1}
\end{align}
\end{lemma}
We can now show that $D_\bx$, $\bx\in [N]^{2M}_{\delta\pm\geq1}$ is a good model for $D_{N,M}^{\pm1}$. For this we only need to show now that
\beq{simpx}{
\Pr(D_\bx\text{ is simple})=\Omega(1).
 }
Again, this follows as in \cite{AFP}. 

Given a tree $H$ with $k$ vertices of in-degrees $y_1,y_2,...,y_k$ and out-degrees $z_1,z_2,...,z_k$ and a fixed vertex $v$ we see that if $\r_H$ is the probability that $D(N_{k_1}(v))=H$  in $D_{\bx}$ then we have
\begin{align}
\r_H&\approx \binom{N}{k-1} \frac{(k-1)!}{Aut(H,o_H)} \sum_{D^-,D^+=k-1}^\infty \nonumber\\
&\sum_{\substack{d_1^-\geq y_1,\ldots,d_k^-\geq y_k
\\d_1^-+\cdots+d_k^-=D^-
\\d_1^+\geq z_1,\ldots,d_k^+\geq z_k
\\d_1^++\cdots+d_k^+=D^+}} 
\prod_{i=1}^k\frac{\l^{d_i^-+d_i^+}}{d_i^-!d_i^+!f_1(\l)^2} \binom{M}{k-1}(k-1)! \prod_{i=1}^k \frac{d_{i}^-!d_i^+!}{(d_i^--y_i)!(d_i^+-z_i)!}\frac{1}{M^{2k-2}}\label{sigma1}\\
\nonumber\\
&\approx\bfrac{N}{M}^{k-1} \frac{\l^{2k-2}}{Aut(H,o_H)f_1(\l)^{2k}} \sum_{\substack{d_1^-+\cdots+d_k^-=D^-\\d_1^++\cdots+d_k^+=D^+}}  \prod_{i=1}^k\frac{\l^{d_i^-+d_i^+-y_i-z_i}}{(d_i^--y_i)!(d_1^+-z_i)!}\nonumber\\
&=\bfrac{N}{M}^{k-1} \frac{\l^{2k-2}}{Aut(H,o_H)f_1(\l)^{2k-2}} \brac{\sum_{D=k-1}^\infty\frac{(k\l)^{D-(k-1)}}{(D-(k-1))!}}^2 \label{sigma2}\\
&\approx \frac{1}{Aut(H,o_H)} \bfrac{N}{M}^{k-1} \lambda^{2k-2} \frac{e^{2k\l}}{f_1(\l)^{2k}}.\nonumber
\end{align}
{\bf Explanation for \eqref{sigma1}:} We use \eqref{local3} to obtain the probability that the in-degrees and out-degrees of $[k]$ are $d_1^-,d_1^+,\ldots,d_k^-,d_k^+$. This accounts for the term $\prod_{i=1}^k\frac{\l^{d_i^-+d_i^+}}{d_i^-!d_i^+!f_1(\l)^2}$. Implicit here is that $d_i^-,d_i^+=O(\log n)$, from \eqref{degconc1}. The contributions to the sum of $D^-,D^+\geq k\log n$ can therefore be shown to be negligible. We use the fact that $k$ is small to argue that w.h.p. $H$ is induced. We choose the vertices, other than $v$ in $\binom{N}{k-1}$ ways and then $\frac{(k-1)!}{Aut(H,o_H)}$ counts the number of copies of $H$ in $K_k$. We then choose the place in the sequence to put these edges in $\binom{M}{k-1}(k-1)!$ ways. Finally note that the probability the $y_i$ occurrences of the $i$th vertex are as claimed is asymptotically equal to $\frac{d_i^-(d_i^--1)\cdots (d_i^--y_i+1)}{M^{z_i}}$ and this explains the factor $\prod_{i=1}^k \frac{d_{i}^-!d_i^+!}{(d_i^--y_i)!(d_i^+-z_i)!}\frac{1}{M^{2k-2}}$.

{\bf Explanation for \eqref{sigma2}:} We use the identity 
\[
\sum_{\substack{d_1,\ldots,d_k\\d_1+\cdots+d_k=D}}\frac{D!}{d_1!\cdots d_k!}=k^D.
\]
It only remains to verify \eqref{nH}. It follows from the above that $\E(\n(H)\mid M,N)=\Omega(N)$. We first condition on a degree sequence \bx\ satisfying \eqref{degconc}. Interchanging two elements in a permutation can only change $\n(H)$ by $O(1)$. We can therefore apply Azuma's inequality to show that 
\beq{Azuma}{
\Pr(|\n(H)-\E(\n(H))|\geq n^{3/5})=O(e^{-\Omega(n^{1/5})}).
} 
(Specifically we can use Lemma 11 of Frieze and Pittel \cite{FP1} or Section 3.2 of McDiarmid \cite{McD}.) This verifies \eqref{nH}.

\appendix
\section{Proof that $D^*$ is Hamiltonian w.h.p.}
The proof can be broken into three parts: suppose that $|V_1^*|=N=N_1+N_2$ where 
\[
N_1=|V_1|\geq N(1-e^{-c/2}).
\]
\begin{enumerate}[(a)]
\item Find a collection $\Pi_1$ of $O(\log N)$ vertex disjoint directed cycles that cover $V_1^*$.
\item Transform $\P_1$ into a collection $\P_2$ of vertex disjoint cycles such that each cycle is of length at least $N_0=\rdup{\frac{200N}{\log N}}$.
\item Break up $\Pi_2$ and re-assemble it as a Hamilton cycle.
\end{enumerate}
\subsection{Constructing $\Pi_1$}
Each vertex of $D^*$ is associated with five blue and five red edges. We randomly select three of each color and make them light and the rest heavy. We now consider the bipartite graph $H$ with bipartition made up of two copies $A,B$ of $V_1^*$ and an edge $\set{v,w}$ iff $(v,w)$ is a light edge. We show that w.h.p. $H$ contains a perfect matching. In the context of $D^*$ this gives us the collection of vertex disjoint directed cycles that cover $V_1^*$. We refer to this as a permutation digraph. We will argue that w.h.p. the number of cycles in the collection is $O(\log N)$. The probability that $H$ has no perfect matching can be bounded by  
\begin{align}
&2\sum_{k=4}^{N/2}\sum_{k_1=0}^k\sum_{k_2=0}^k\binom{N_1}{k_1}\binom{N_1}{k_2}\binom{N_2}{k-k_1} \binom{N_2}{k-k_2} \bfrac{k_2}{N_1}^{3k} \brac{1-\frac{k_1}{N_1}}^{3(N-k)}\label{A1}\\
&\leq 2\sum_{k=4}^{N/2}\sum_{k_1=0}^k\sum_{k_2=0}^k\binom{N}{k}\binom{N}{k} \bfrac{k}{N_1}^{3k} \leq 2\sum_{k=4}^{N/2} k^2\bfrac{eN}{k}^{2k} \bfrac{k}{N_1}^{3k} \nonumber \\
&\leq  2\sum_{k=4}^{N/2} k^2\bfrac{e^2 k}{(1-e^{-c/2})N}^k  =o(1)\nonumber.  
\end{align}  
{\bf Explanation for \eqref{A1}:} we employ Hall's theorem. We choose a set $S\subseteq A$ of size $k\leq N/2$ and a set $T\subseteq B$ also of size $k$. (No need to make $|T|=k-1$ here.) We let $k_1=|S\cap V_1|$ and $k_2=|T\cap V_1|$. The number of ways of choosing these sets is given by the product of binomial coefficients. We then estimate the probability that $T\supseteq N(S)$. Each vertex in $S\cap A$ has probability at most $\bfrac{k_2}{N_1}^3$ of choosing all of its neighbors in $V_1\cap T$, explaining the factor $\bfrac{k_2}{N_1}^{3k}$. Each vertex in $B\setminus T$ has probability $ \brac{1-\frac{k_1}{N_1}}^3$ of not choosing any neighbors in $V_1\cap S$, explaining the term  $\brac{1-\frac{k_1}{N_1}}^{3(N-k)}$.

This deals with $k\leq N/2$ and if $k>N/2$ then $B\setminus T$ and $A\setminus S$ can take the place of $S,T$ respectively..

We now consider the number of cycles in cycle cover induced by a matching in $H$. Suppose we write $M=\set{(m(i),i):i\in B}$ for some permutation $m$ of $A$. Further let $A=A_1\cup A_X$ where $A_1=\set{a_1,a_2,\ldots,a_{N_1}}$ corrsponds to $V_1$ and $A_X$ corresponds to $\vX^*$. We assume an analogous decomposition for $B$. Given a permutation $m$ we let $B_X(m)=\set{b\in B:m(b)\in A_X}\subseteq B_1$. The set inclusion follows from the fact that vertices in $A_X$ only have neighbors in $B_1$. Suppose now that we assume after re-labelling that that $A,B$ are disjoint copies of $[N_1]$ and that $B_X(m),A_X$ are disjoint copies of $[N_2]$. Thus $m$ induces a permutation of $[N_2]$ and a permutation of $[N_2+1,N]$. We claim that conditional on this that $m$ induces uniform random permutations on these two sets. Suppose now that $m_1,m_2$ are two permutations that satisfy $m_i([N_2])=[N_2]$ for $i=1,2$. For a permutation $\p$ of $A$ that satisfies $\p([N_2]))=[N_2]$ and graph $H$ we let $\p(H)$ be obtained from $H$ by replacing edge $\set{i,j}$ by $\set{\p(i),j}$. We note that $H$ and $\p(H)$ have the same distribution. But then where $\p(a)=m_2(m_1^{-1}(a))$ for $a\in A$ we have
\beq{mA1}{
\Pr(m(H)=m_1)=\Pr(m(\p(H))=m_2)=\Pr(m(H)=m_2),
}
justifying our uniformity claim.

Now a uniform random permutation on a set of size $M$ has $O(\log M)$ cycles w.h.p. It follows that w.h.p. the number of cycles induced by the matching constructed in $H$ has $O(\log N)$ cycles as claijmed previously.
\subsection{Constructing $\Pi_2$}
We now show how to boost the minimum cycle size to at least $N_0$. We partition the cycles of the permutation digraph $\Pi_1$ into sets SMALL and LARGE, containing cycles $C$ of length $|C| < N_0$ and $|C| \geq n_0$ respectively. We define a Near Permutation Digraph (NPD) to be a digraph obtained from a permutation digraph by removing one edge. Thus an NPD $\Gamma$ consists of a path $P(\Gamma)$ plus a permutation digraph $PD(\Gamma)$ which covers $[n]\setminus V(P(\Gamma))$.

We now give an informal description of a process which removes a small cycle $C$ from a {\em current} 
permutation digraph $\Pi$. We start by choosing an (arbitrary) edge $(v_0,u_0)$ of $C$ and delete it to obtain an NPD $\Gamma_0$ with $P_0=P(\Gamma_0)\in {\cal P}(u_0,v_0)$, where ${\cal P}(x,y)$ denotes the set of paths from $x$ to $y$ in $D$. The aim of the process is to produce a {\em large} set $S$ of NPD's such that for each $\Gamma\in S$, (i) $P(\Gamma)$ has a least $N_0$ edges and (ii) the small cycles of $PD(\Gamma)$ are a subset of the small cycles of $\Pi$. We will show that {\bf whp} the endpoints of one of the $P(\Gamma)$'s can be joined by an edge to create a permutation digraph with (at least) one less small cycle. 

We have so far used six of the edges available at each vertex of $D^*$. We now let $D_4$ denote the digraph associated with an used fourth in- and out-edge associated with each vertex of $D^*$. Each vertex $v\in V^*$ will be associated with a random in-neighbor $in_4(v)$ and a random out-neighbor $out_4(v)$. 

The basic step in an {\em Out-Phase} of this process is to take an NPD $\Gamma$ with $P(\Gamma)\in {\cal P}(u_0,v)$ and to examine the edges of $D_4$ leaving $v$ i.e. edges going {\em out} from the end of the path. Let $w$ be the terminal vertex of such an edge and assume that $\Gamma$ contains an edge $(x,w)$. Then $\Gamma'=\Gamma\cup\{(v,w)\}\setminus\{(x,w)\}$ is also an NPD. $\Gamma'$ is acceptable if (i) $P(\Gamma')$ contains at least $N_0$ edges and (ii) any new cycle created (i.e. in $\Gamma'$ and not $\Gamma$) also has at least $N_0$ edges.

If $\Gamma$ contains no edge $(x,w)$ then $w=u_0$. We accept the edge if $P$ has at least $N_0$ edges. This would (prematurely) end an iteration, by closing a cycle, although it is unlikely to occur.

We do not want to look at very many edges of $D_4$ in this construction and we build a tree $T_0$ of NPD's in a natural breadth-first fashion where each non-leaf vertex $\Gamma\in T_0$ gives rise to NPD children $\Gamma'$ as described above. The construction of $T_0$ ends when we first have $\nu=\rdup{\sqrt{N\log N}}$ leaves. The construction of $T_0$ constitutes an Out-Phase of our procedure to eliminate small cycles. Having constructed $T_0$ we need to do a further {\em In-Phase}, which is similar to a set of Out-Phases.

Then w.h.p. we close at least one of the paths $P(\Gamma)$ to a cycle of length at least $N_0$. If $|C|\geq 4$ and this process fails then we try again with a different independent edge of $C$ in place of $(u_0,v_0)$.

We now increase the the formality of our description. We start Phase 2 with a permutation digraph $\Pi_0$ and a general iteration of Phase 2 starts with a permutation digraph $\Pi$ whose small cycles are a subset of those in $\Pi_0$. Iterations continue until there are no more small cycles. At the start of an iteration we choose some small cycle $C$ of $\Pi$. There then follows an Out-Phase in which we construct a tree $T_0=T_0(\Pi,C)$ of NPD's as follows: the root of $T_0$ is $\Gamma_0$ which is obtained by deleting an edge $(v_0,u_0)$ of $C$.

We grow $T_0$ to a depth at most $\rdup{1.5\lo}$. The set of nodes at depth $t$ is denoted by $S_t$.
\\
Let $\Gamma \in S_t$ and $P=P(\Gamma)\in {\cal P}(u_0,v)$. A {\em potential} child $\Gamma'$ of $\Gamma$, at depth $t+1$ is defined as follows.

Let $w$ be the terminal vertex of an edge directed from $v$ in $D_4$.
\\
{\em Case 1.} $w$ is a vertex of a cycle $C' \in PD(\Gamma)$ with
edge $(x,w) \in C'$. Let $\Gamma'=\Gamma\cup\{(v,w)\}\setminus \{(x,w)\}$.
\\
{\em Case 2}. $w$ is a vertex of $P(\Gamma)$. Either $w = u_0$, or
$(x,w)$ is an edge of $P$. In the former case $\Gamma\cup\{(v,w)\}$ is a permutation digraph $\Pi'$ and in the latter case we let $\Gamma'=\Gamma\cup\{(v,w)\}\setminus \{(x,w)\}$.

In fact we only admit to $S_{t+1}$ those $\Gamma'$ which satisfy the following conditions. We define a set $W$ of {\em used} vertices. Initially all vertices are {\em unused} i.e. $W=\emptyset$. Whenever we examine an edge $(v,w)$, we add both $v$ and $w$ to $W$. So if $v\not\in W$ then $out_4(v)$ is still unconditioned and $in_4(v)$ is a random member of a set $U\supseteq V^*\setminus W$. We do not allow $|W|$ to exceed $N^{3/4}$.
\begin{enumerate}[{\bf C(i)}]
\item The new cycle formed (Case 2 only) must have at least $N_0$ vertices, and the path formed (both cases) must either be empty or have  at least $N_0$ vertices. When the path formed is empty we close the iteration and if necessary start the next with $\Pi'$.
\item $x,w \not \in W$ .
\end{enumerate}
An edge $(v,w)$ which satisfies the above conditions is described as {\em acceptable}.

{\red We also let $S^1_t=S_t \cap V_1$ and $S_2^t=S_t\setminus S_1^t$. } 
\begin{lemma}\label{a1}
Let $C \in$ SMALL. Then, where $\nu=\rdup{\sqrt{N\log N}}$, 
$$\Pr(\exists t <\rdup{\log_{1.9}\nu+1000\log \log N}\text{ such that  }|S_{t}| \in[\nu,3\nu] ) =1-O((\log \log N)^3/ \log N).$$ 
\end{lemma}
\begin{proof}
We assume we stop an iteration, in mid-phase if necessary, when $|S_t| \in[\nu,3\nu]$. Let us consider a generic construction in the growth of $T_0$. Thus suppose we are extending from $\Gamma$ and $P(\Gamma)\in{\cal P}(u_0,v)$. 

We consider $S_{t+1}$ to be constructed in the following manner: we first examine $out_4(v), v\in S_t$ in the order that these vertices were placed in $S_t$ to see if they produce acceptable edges. We then add in those vertices $x\not\in W$ which arise from $(x,w)$ with $v=in_4(w)\in S_t,w\not\in W$, (to avoid conditioning problems).

Let $Z(v)$ be the indicator random variable for $(v,out_4(v))$ being unacceptable and let $Z_t=\sum_{v\in S_t}Z(v)$. If $Z(v)=1$ then either (i) $out_4(v)$ lies on $P(\Gamma)$ and is too close to an endpoint; this has probability bounded above by $2N_0/|V_1|\leq 401/\log N$, or (ii) the corresponding vertex $x$ is in $W$; this has probability bounded above by $N^{3/4}/|V_1|\leq 2N^{-1/4}$, or (iii) $out_4(v)$ lies on a small cycle. Now in a random permutation the expected number of vertices on cycles of length at most $N_0$ is precisely $N_0$ (\cite{K}). Thus, by the Markov inequality, w.h.p. $\Gamma_0$ contains at most $N_1\log\log N_1/(2\log N_1)+N_2\log\log N_2/(2\log N_2)$ vertices on small cycles. Condition on this event. Then $\Pr(Z(v)=1)\leq 2\log\log N/\log N$ regardless of the history of the process and so $Z_t$ is stochastically dominated by $B(|S_t|,2\log\log N/\log N)$.

Next let $X(v)$ denote the number of vertices $w$ in $V^*\setminus W$ such that $in_4(w)=v$, $x\not\in W$ where $(v,w)$ is acceptable and $(x,w)\in \Gamma$ (if there is no such $x$ then the iteration can end early.) Let $X_t=\sum_{v\in S_t}X(v)$. Now assuming $|W|\leq N^{3/4}$ we see that there are $N'=N_1-O(N\log\log N/\log N)$ vertices $w$ which would produce an acceptable edge provided {\red $v=in_4(w)\in S_t^1$}. For these vertices $in_4(w)$ is a random choice from a set which contains ${\red S_t^1}$ and so $X_t$ stochastically dominates $B(N',{\red |S_t^1|}/N)$. 

Summing $1-Z(v)+X(v)$ over $v\in S_t$ might seem to overestimate $\card{S_{t+1}}$. In principle we should subtract off the number $Y_t$ of vertices of $S_{t+1}$ that are counted more than once in this sum. But these arise in two ways. First there are the pairs $v_1,v_2\in S_t$ with $out_4(v_1)=out_4(v_2)$. Suppose we examine $v_1$ before $v_2$. Then when we examine $v_2$ we find that $out_4(v_2)\in W$ and so we do not get a contribution to $S_{t+1}$. Secondly there is the possibility of their being $v_1,v_2\in S_t$ and $w$ such that $w=out_4(v_1)$ and $v_2=in_4(w)$. But in this case $w$ will only be counted once as $w\in W$ when it is time for $in_4(w)$ to be examined. We can then write 
$$|S_{t+1}|  =  |S_t|-Z_t+X_t.$$
Now let $t_0=\rdup{1000\log\log N}$, $t_1=10t_0$, $t_2=\rdup{\log_{1.9}\nu+1000\log \log N}$, $s_0=\rdup{1000\log \log N}$ and $s_1=\rdup{1000\log N}$. 
\begin{description}
\item[(a)] $\Pr(\exists t\leq t_0: |S_t|\leq s_0\mbox{ and }Z_t>0)=O((\log\log N)^3/\log N)$
\item[(b)] {\red $\Pr( |\cup_{t\leq t_0} S_t^1|< 0.99
|\cup_{t\leq t_0} S_{t}| \mid |S_t| \leq s_0 \mbox{ for }t\leq t_0)=O((\log\log N)^3/\log N)$}.
\item[(c)] $\Pr(\sum_{t=1}^{t_0}X_t\leq s_0 \mid S_t\neq \emptyset {\red \mbox{ and }|S_t| \leq s_0 \mbox{ for } t\leq t_0})=O((\log\log N)^3/\log N)$.
\item[(d)] {\red $\Pr(\exists t\leq t_1: |S_{t+1}^1|< 0.99
| S_{t+1}| \mid S_t \geq 500\log \log n)=O(1/\log N)$}.
\item[(e)] $\Pr(\exists t\leq t_1: 500\log\log N\leq |S_t|\leq s_1\mbox{ and }Z_t>X_t/100)=O(1/\log N)$.
\item[(f)] $\Pr(\exists t\leq t_1: X_t<|S_t|/2\mid\ \card{S_t}\geq  500\log\log N)=O(1/\log N)$.
\item[(g)] $\Pr(\exists t\leq t_1: |S_t|\leq s_1\mbox{ and }X_t\geq 2s_1)=O(N^{-2})$.
\item[(h)] {\red $\Pr(\exists t_1\leq t\leq t_2: |S_{t+1}^1|< 0.99
| S_{t+1}| \mid S_t \geq s_1)=O(N^{-2})$}.
\item[(i)] $\Pr(\exists t\leq {\red t_2}:|S_t|\geq s_1\mbox{ and }|X_t-Z_t-|S_t||\geq |S_t|/10)=O(N^{-2}).$
\end{description}
{\bf Explanations:-} we use the following standard inequalities for the tails of the binomial distribution:
\begin{eqnarray}
\Pr(\card{B(n,p)-np}\geq \epsilon np) & \leq & 2e^{-\epsilon^ 2np/3},\hspace{.25in}0\leq \epsilon \leq 1, \label{Ch} \\
\Pr(B(n,p)\geq anp) & \leq & (e/a)^{anp}. \label{LD}
\end{eqnarray}
We let Let ${\cal E}_x,x\in \{a,b,\ldots ,i\}$ be the low probability events described in (a)-(i) above.
\begin{description}
\item[(a)] $\Pr(Z_t>0\mid\ \card{S_t}\leq 500\log\log N)=O((\log\log N)^2/\log N)$ by the Markov inequality.
\item[(b)] {\red  Conditioned on ${\cal E}_a$ we have that  
$|\cup_{t\leq t_0} S_{t}| \geq t_0$ and $Z_t=0$ for $t\leq t_0$. Let $v_1,v_2,...$ be the order in which the vertices in $\cup_{t\leq t_0} S_{t}$ are examined. At step $i$ with $w=out_4(v_i)$ we updated  $\Gamma'=\Gamma\cup\{(v_i,w)\}\setminus \{(x,w)\}$ and added $x$ to 
$\cup_{t\leq t_0} S_{t}$. $x$ belongs to $V_1$ with probability $(1+o(1))|N_1|/N > 0.999$. The rest follows from (\ref{Ch}). }
\item[(c)]{\red Conditioned on ${\cal E}_a\cap {\cal E}_b$ we have that $|\cup_{t\leq t_0} S_t^1|\geq 0.99t_0$. Thus $\sum_{t=1}^{t_0}X_t$ dominates $B(0.99t_0N',1/N)$.}
\item[(d)] {\red Similar to (b).}
\item[(e)] Condition on $\card{S_t}=s\geq  500\log\log N$ and ${\cal E}_d$. Then $Z_t>X_t/100$ implies either that (i) $X_t\leq s/10 \leq 0.99|S_t^1|/10$ or (ii) $Z_t>10s$. 
Both of these events have probability $O(1/(\log N)^3)$.
\item[(f)] Immediate from (\ref{Ch}).
\item[(g)] Immediate from (\ref{Ch}) and (\ref{LD}).
\item[(h)] {\red Similar to (b).}
\item[(i)] Similar to (c).
\end{description}
Assume the occurrence of $\bigcap_x\bar{{\cal E}}_x$. Then $\bar{{\cal E}}_a\cap\bar{{\cal E}}_c $ implies that $|S_t|$ reaches size at least $500\log\log N$ before $t$ reaches $t_0 {\red +1}$. Once this happens, $\bar{{\cal E}}_e \cap\bar{{\cal E}}_f$ implies that $|S_t|$ then grows geometrically with $t$ up to time $t_1$ at a rate of at least 1.49. Together with $\bar{{\cal E}}_g$ this proves that at some stage between 1 and $t_1$, $|S_t|$ reaches a size in the range $[s_0,3s_0]$. $\bar{{\cal E}}_f$ then implies that $|S_t|$ increases at a rate $\lambda\in[1.9,2.1]$ from then on. The lemma follows.
\end{proof}

The total number of vertices added to $W$ in this way throughout the whole of Phase 2 is $O(\nu|SMALL|)=o(N^{3/4})$. (As we see later, we try this process once for $C\in SMALL,|C|\leq 3$ and once or twice for $C\in SMALL,|C|\geq 4$.)

Let $t^*$ denote the value of $t$ when we stop the growth of $T_0$. At this stage we have leaves $\Gamma_i$,
for $i = 1,\ldots,\nu$, each with a path of length at least $N_0$, (unless we have already successfully made a cycle). We now execute an In-Phase. This involves the construction of trees $T_i,i=1,2,\ldots \nu$. Assume that $P(\Gamma_i)\in {\cal P}(u_0,v_i)$. We start with $\Gamma_i$ and build $T_i$ in a similar way to $T_0$ except that here all paths generated end with $v_i$. This is done as follows: if a current NPD $\Gamma$ has $P(\Gamma)\in {\cal P}(u,v_i)$ then we consider adding an edge $(w,u)\in D_4$ and deleting an edge $(w,x)\in \Gamma$. Thus our trees are grown by considering edges directed into the start vertex of each $P(\Gamma)$ rather than directed out of the end vertex. Some technical changes are necessary however. 

We consider the construction of our $\nu$ trees in two stages. First of all we grow the trees only enforcing condition  C(ii) of success and thus allow the formation of small cycles and paths. We try to grow them to depth $t_2$. The growth of the $\nu$ trees can naturally be considered to occur simultaneously. Let $L_{i,\ell }$
denote the set of start vertices of the paths associated with the nodes at depth $\ell$ of the $i$'th tree, $i=1,2\ldots ,\nu, \ell = 0,1,\ldots,t_2$. Thus $L_{i,0}=\{ u_0\}$ for all $i$. We prove inductively that $L_{i,\ell}=L_{1,\ell}$ for all $i,\ell$. In fact if $L_{i,\ell}=L_{1,\ell}$ then the acceptable $D_4$ edges have the same set of initial vertices and since all of the deleted edges are $D_a$-edges (enforced by C(ii)) we have $L_{i,\ell +1}=L_{1,\ell +1}$.

The probability that we succeed in constructing trees $T_1,T_2,\ldots T_{\nu}$ is, by the analysis of Lemma 3, $1-O((\log \log N)^3/\log N)$. Note that the number of nodes in each tree is $O(2.1^{t_2+1})=O( N^{.74\ldots})$.

We now consider the fact that in some of the trees some of the leaves may have been constructed in violation of C(i). We imagine that we prune the trees $T_1,T_2,\ldots T_\nu$ by disallowing any node that was constructed in violation of C(i). Let a tree be BAD if after pruning it has less than $\nu$ leaves and GOOD otherwise. Now an individual pruned tree has been constructed in the same manner as the tree $T_0$ obtained in the Out-Phase. (We have chosen $t_2$ to obtain $\nu$ leaves even at the slowest growth rate of 1.9 per node.) Thus
$$\Pr(T_1 \mbox{ is BAD})=O\bfrac{(\log \log N)^3}{\log N}$$
and 
$$\E(\mbox{number of BAD trees}) = O\bfrac{\nu(\log \log N)^3}{\log N}$$
and
$$\Pr(\exists \geq \nu/2 \mbox{ BAD trees})=O\bfrac{(\log \log N)^3}{\log N}.$$
Thus 
\begin{align*}
&\Pr(\exists <\nu/2 \mbox{ GOOD trees after pruning})\\
&\leq \Pr(\mbox{failure to construct } T_1,T_2,\ldots T_\nu)+\Pr(\exists \geq \nu/2 \mbox{ BAD trees}) \\
&= O\bfrac{(\log \log N)^3}{\log N}. 
\end{align*}
Thus with probability 1-$O((\log \log N)^3/\log N)$ we end up with $\nu/2$ sets of $\nu$ paths, each of length at least $100n/\log N$ where 
the $i$'th set of paths all terminate in $v_i$. 
{\red From these paths keep only those whose other endpoint $u$ lies in $V_1$. Then, similarly to the proof of property (h) in Lemma \ref{a1}, 
w.h.p. from each set we keep at least $0.99 \nu$ paths.
}
The $in_4(v_i)$ are still unconditioned and hence
\[
\Pr(\mbox{no $D_4$ edge closes one of these paths})  \leq  \left( 1-\frac{{\red 0.99\nu}}{n}\right) ^{\nu/2} =  O(N^{-1/2}).
\]
Consequently the probability that we fail to eliminate a particular small cycle $C$ after breaking an edge is
$O((\log \log N)^3/\log N)$. If $|C|\geq 4$ then we try once or twice using independent edges of $C$ and so the probability we fail to eliminate 
a given small cycle $C$ is certainly $O(((\log\log N)^3/\log N)^2)$ for $|C|\geq 4$ (remember that we calculated all probabilities conditional on 
previous outcomes and assuming $|W|\leq N^{3/4}$.)

Now the number of cycles of length 1,2 or 3 in $D_a$ is asymptotically Poisson with mean 11/6 and so there are fewer than $\log\log N$ w.h.p. Hence, since {\bf whp } $|C|=O(\log N)$, 
\begin{lemma}
The probability that Phase 2 fails to produce a permutation digraph with minimal cycle length at least $N_0$ is $o(1)$.
\end{lemma}
At this stage we have shown that $D^*$ almost always contains a permutation digraph $\Pi_2$ in which the minimum cycle length is at least $N_0$. We shall refer to $\Pi_2$ as the {\em Phase 2} permutation digraph.
\subsection{Re-assembly}
Let $D_5$ be the 1-in,1-out digraph left unused by the construction in the previous two sections. We will use the edges of $D_5$ to break-up and re-assemble the cycles of $\Pi_2$ into a Hamilton cycle. Let $C_1,C_2,\ldots,C_k$ be the cycles of $\Pi^{*}$, and let $c_i = |C_i\cap V_1|, \; c_1 \leq c_2 \leq \cdots \leq c_k$. Note that $\vX^*$ is an independent set of $D^*$ and so at least half the vertices of each $C_i$ are in $V_1$. If $k=1$ we can skip this phase, otherwise let $a = \frac{N}{\log N}$. For each $C_i$ we consider selecting a set of $m_i = 2 \lfloor \frac{c_i}{a} \rfloor + 1$ vertices $v\in C_i\cap  V_1$, and deleting the edge $(v,u)$ in $\Pi^{*}$. Let $m = \sum_{i=1}^k m_i$ and re-label (temporarily) the broken edges as $(v_i ,u_i), i \in [m]$ as follows: in cycle $C_i$ identify the lowest numbered vertex $x_i$ which loses a cycle edge directed out of it. Put $v_1=x_1$ and then go round $C_1$ defining $v_2,v_3,\ldots v_{m_1} $ in order. Then let $v_{m_1+1}=x_2$ and so on. We thus have $m$ path sections $P_j\in {\cal P}(u_{\f(j)},v_j) $ in $\Pi^{*}$ for some permutation $\f$.  We see that $\f$ is an even permutation as all the cycles of $\f$ are of odd length.

It is our intention to rejoin these path sections of $\Pi^{*}$ to make a Hamilton cycle using $D_b$, if we can. Suppose we can. This defines a permutation $\r$ where $\r (i) = j$ if $P_i$ is joined to $P_j$ by $(v_i,u_{\phi (j)})$, where $\r\in H_m$ the set of cyclic permutations on $[m]$. We will use the second moment method to show that a suitable $\r$ exists w.h.p. A technical problem forces a restriction on our choices for $\r$. This will produce a variance reduction in a second moment calculation.

Given $\r$ define $\lambda=\f\r$. In our analysis we will restrict our attention to $\r\in R_{\f} = \{ \r \in H_m : \f \r \in H_m \}$. If $\r\in R_{\f}$ then we have not only constructed a Hamilton cycle in $\Pi^{*}\cup D_b$, but also in the {\em auxillary digraph} $\LL$, whose edges are $(i, \lambda(i))$.
\begin{lemma}
$(m-2)! \leq |R_{\f}| \leq (m-1)!$
\end{lemma}
\begin{proof}
We grow a path $1,\lambda (1), \gk{2},\ldots,\gk{r}\ldots $ in $\LL$, maintaining feasibility in the way we join the path sections of $\Pi^{*}$ at the same time.

We note that the edge $(i,\lambda(i))$ of $\LL$ corresponds in $D_b$ to the edge $(v_i,u_{\f\r(i)})$. In choosing $\lambda(1)$ we must avoid not only 1 but also $\f(1)$ since $\lambda(1)=1$ implies $\r(1)=1$. Thus there are $m-2$ choices for $\lambda (1)$ since $\f (1) \neq 1$ from the definition of $m_1$. 

In general, having chosen $\lambda(1),\gk{2},\ldots,\gk{r},1\leq r\leq m-3$ our choice for $\gk{r+1}$ is restricted to be different from these choices and also  1 and $\ell$ where $u_\ell$ is the initial vertex of the path terminating at $v_{\lambda^r (1)}$ made by joining path sections of $\Pi^{*}$. Thus there are either $m-(r+1)$ or $m-(r+2)$ choices for $\gk{r+1}$ depending on whether or not $\ell=1$.

Hence, when $r = m-3$, there {\em may} be only one choice for $\gk{m-2}$, the vertex $h$ say. After adding this edge, let the remaining isolated vertex of $\LL$ be $w$. We now need to show that we can complete 
$\lambda $, $\r$ so that $\lambda , \r \in H_m$.

Which vertices are missing edges in $\LL$ at this stage? Vertices $1,w$ are missing in-edges, and $h,w$ out-edges. Hence the path sections of $\Pi^{*}$ are joined so that either
\[
u_1 \rightarrow v_h ,\; \; u_w \rightarrow v_w \; \; \;\mbox{   or   } \; \; \;u_1 \rightarrow v_w ,\; \; u_w \rightarrow v_h.
\]
The first case can be (uniquely) feasibly completed in both $\LL$ and $D$ by setting $\lambda (h)=w,\lambda (w)=1$. Completing the second case to a cycle in $\Pi^{*}$ means that 
\begin{equation}
\label{poop}
\lambda = ( 1, \lambda (1),\ldots, \gk{m-2} ) (w)
\end{equation}
and thus $\lambda \not \in H_m$. We show this case cannot arise.

$\lambda = \f \r$ and $\f$ is even implies that $\lambda$ and $\r$ have the same parity. On the other hand $\r \in H_m$ has a different parity to $\lambda$ in (\ref{poop}) which is a contradiction.

Thus there is a (unique) completion of the path in $\LL$. 
\end{proof}

Let $H$ stand for the union of the permutation digraph $\Pi^{*}$ and $D_5$. We finish our proof by proving
\begin{lemma}
$\Pr( H$ does not contain a Hamilton cycle ) = $o(1)$.
\end{lemma}
{\em Proof.}
Let $X$ be the number of Hamilton cycles in $H$ obtainable by deleting edges as above, rearranging the path sections generated by $\f$ according to those $\r \in R_{\f}$ and if possible reconnecting all the sections using edges of $D_5$. We will use the inequality
\begin{equation}
\label{eq?}
\Pr(X>0)\geq \frac{\E(X)^2}{\E(X^2)}.
\end{equation}
Probabilities in (\ref{eq?}) are thus with respect to the space of $D_5$ choices.

Now the definition of the $m_i$ yields that
$$\frac{2N}{a}-k\leq m\leq \frac{2N}{a}+k$$
and so
$$(1.99)\log N\leq m\leq (2.01)\log N.$$
Also 
$$k\leq \frac{\log N}{200},\,m_i\geq 199\mbox{ and }\frac{c_i}{m_i}\geq \frac{a}{2.01},\ \ \ \ \ 1\leq i\leq k.$$

Let $\Omega$ denote the set of possible cycle re-arrangements. $\omega\in\Omega$ is a {\em success} if $D_5$ contains the edges needed for the associated Hamilton cycle. 
{\red Let $b_i$ be the number of deleted edges $(v_i,u_i)$ with $u_i \notin V_1$ and $b=\sum_{i=1}^k b_i$.
Observe that if $u_i\in V_1$ then $(v_i,u_i)\in E(D_5)\setminus E(D_4)$ with probability $1-\big(1-\frac{1}{N_1}\big)^2$ while if  $u_j\notin V_1$ then $(v_i,u_j)\in E(D_5)\setminus E(D_4)$ with probability $ \frac{1}{N_1}$.

For a fixed $\a>0$ we have 
\[
ne^{-c/2}\geq N-N_1\geq b\geq \sum_{j:b_j\geq\a|C_j|}b_j\geq \a\sum_{j:b_j\geq\a|C_j|}|C_j|.
\]
Putting $\a=10^{-3}$ we see that at most $1000ne^{-c/2}\leq e^{-c/3}N $ vertices lie on a cycle $C_i$ with more than $0.001|C_i|$ vertices that do not lie in $V_1$. Therefore $b$ is stochasticly dominated by  
$(1+o(1))(e^{-c/3}m+ Bin ( (1-e^{-c/3})m,10^{-3})$. Hence  $\mathbb{P}(b > 0.01 m)=o(1)$.}
Thus, 
\begin{eqnarray}
\E(X) & = & \sum_{\omega\in\Omega}\Pr(\omega\mbox{ is a success}) \nonumber \\
& = & \sum_{\omega\in\Omega} \left(1-\left(1-\frac{1}{N_1}\right)^2\right)^{m-b(\omega)}\bfrac{1}{N_1}^{b(\omega)}  \nonumber\\
& \geq & (1-o(1))\left(\frac{2}{N_1}\right)^m {\red 2^{-0.01m} \cdot  \mathbb{P}(b \leq 0.01 m) }(m-2)!\prod_{i=1}^k \binom{c_i}{m_i} \nonumber  \\
& \geq & \frac{1-o(1)}{m \sqrt{m}} \left( \frac{2m}{eN_1} \right)^m
\prod_{i=1}^k \left(\left( \frac{c_ie^{1-1/12m_i}}{m_i^{1+(1/2m_i)}} \right)^{m_i}\left(\frac{1-2m_i^2/c_i}{ \sqrt{2\pi} }\right)\right) {\red 2^{-0.01m} } \nonumber\\
& \geq & \frac{(1-o(1))(2\pi)^{-m/398}e^{-k/12}}{m \sqrt{m}} \left( \frac{2m}{eN_1} \right)^m
\prod_{i=1}^k \left( \frac{c_ie}{(1.02)m_i} \right)^{m_i} {\red 2^{-0.01m} } \nonumber \\
& \geq &  \frac{(1-o(1))(2\pi)^{-m/398}}{n^{1/1200}m \sqrt{m}}\left( \frac{2m}{eN_1} \right)^m \left(\frac{ea}{ 2.01\times 1.02}\right)^m {\red 2^{-0.01m} } 
\nonumber \\
& \geq & \frac{(1-o(1))(2\pi)^{-m/398}}{N_1^{1/1200}m \sqrt{m}}\left( \frac{3.98}{2.0502} \right)^m {\red 2^{-0.01m} }\nonumber\\
& \geq & N_1^{1.3}\label{noo2}. 
\end{eqnarray}
Let $A ,A^{\prime}$ be two sets of selected edges which have been deleted in $\Pi_2$ and whose path sections have been rearranged into Hamilton cycles according to $\r , \r^{\prime}$ respectively. Let $B,B'$ be the corresponding sets of edges which have been added to make the Hamilton cycles. What is the interaction between these two Hamilton cycles? 

Let $s=|A\cap A'|$ and $t=|B\cap B'|$. Now $t\leq s$ since if $(v,u)\in B\cap B'$ then there must be a unique $(\tilde{v},u)\in A\cap A'$ which is the unique $\Pi^{*}$-edge into $u$. We claim that $t=s$ implies $t=s=m$ 
and $(A,\rho )=(A',\rho ')$. (This is why we have restricted our attention to $\r \in R_{\f}$.) Suppose then that $t=s$ and $(v_i,u_i)\in A\cap A'$. Now the edge $(v_i,u_{\lambda (i)})\in B$ and since $t=s$ this edge must also be in $B'$. But this implies that $(v_{\lambda (i)},u_{\lambda (i)})\in A'$ and hence in $A\cap A'$. Repeating the argument we see that $(v_{\lambda ^k(i)},u_{\lambda ^k(i)})\in A\cap A'$ for all $k\geq 0$. But $\lambda$ is cyclic and so our claim follows.

We adopt the following notation. Let $<s,t>$ denote $|A\cap A'|=s$ and $|B\cap B'|=t$. So
\begin{eqnarray}
\E(X^2) &\leq & \E(X) + (1+o(1))\sum_{A\in\Omega} \left( \frac{2}{N_1}\right) ^m \sum_{\substack{ {\red A'\in}
\Omega \\B'\cap B=\emptyset}}      \left( \frac{2}{N_1} \right) ^m \nonumber \\
& & +(1+o(1))\sum_{A\in \Omega} \left( \frac{2}{N_1} \right) ^m \sum_{s=2}^m
\sum_{t=1}^{s-1}\sum_{\substack{ {\red A'\in} \Omega\\ <s,t>}}     \left( \frac{2}{N_1} \right) ^{m-t} \nonumber \\
&=&\E(X)+E_1 + E_2 \; \mbox{ say}. \label{eq???}
\end{eqnarray}
Clearly 
\begin{equation}
\label{eq????}
E_1 \leq (1+o(1))\E(X)^2. 
\end{equation}
For given $\r$, how many $\r^{\prime}$ satisfy the condition $<s,t>$? Previously $|R_{\f}| \geq (m-2)!$ and now given $<s,t>$, $|R_{\f}(s,t)| \leq (m-t-1)!$, (consider fixing $t$ edges of $\LL^{\prime}$).
\\
Thus
\[
E_2 \leq \E(X)^2 \; \sum_{s=2}^m \sum_{t=1}^{s-1} \binom{s}{t}\left[\sum_{\scriptstyle \sigma_1+ \cdots +\sigma_k =s}
\prod_{i=1}^k \frac{\binom{m_i}{\sigma_i}\binom{c_i - m_i}{m_i - \sigma_i}}     { \binom{c_i}{m_i}}
\right] \frac{(m-t-1)!}{(m-2)!} \left( \frac{N_1}{2} \right) ^t .
\]
Now
\begin{eqnarray*}
\frac{\binom{c_i-m_i}{m_i-\sigma_i}}{\binom{c_i}{m_i}} & \leq & \frac{\binom{c_i}{m_i-\sigma_i}}{ \binom{c_i}{m_i}} \\
& \leq & ( 1+o(1))\left( \frac{m_i}{c_i} \right) ^{\s _i}\exp\left\{-\frac{\sigma_i(\sigma_i-1)}{2m_i}\right\} \\
& \leq & ( 1+o(1))\left( \frac{2.01}{a} \right) ^{\s _i}\exp\left\{-\frac{\sigma_i(\sigma_i-1)}{2m_i}\right\}
\end{eqnarray*}
where the $o(1)$ term is $O((\log N)^3/N)$. Also 
$$\sum_{i=1}^k\frac{\sigma_i^2}{2m_i}\geq \frac{s^2}{2m}\hspace{.25in}\mbox{for }\sigma_1+\cdots \sigma_k=s,$$
$$\sum_{i=1}^k\frac{\sigma_i}{2m_i}\leq \frac{k}{2},$$
and
$$\sum_{\scriptstyle \sigma_1+ \cdots +\sigma_k =s} \prod_{i=1}^k 
\binom{m_i}{\sigma_i}=\binom{m}{s}.$$
Hence
\begin{eqnarray}
\frac{E_2}{\E(X)^2} & \leq &
(1+o(1))e^{k/2}\sum_{s=2}^m \sum_{t=1}^{s-1}\binom{s}{t}\exp\left\{-\frac{s^2}{2m}\right\}\left( \frac{2.01}{a} \right) ^s \binom{m}{s}\frac{(m-t-1)!}{(m-2)!}  
\left( \frac{N_1}{2}\right) ^t \nonumber \\
& \leq & (1+o(1))N^{.005}\sum_{s=2}^m \sum_{t=1}^{s-1} \binom{s}{t}\exp\left\{-\frac{s^2}{2m}\right\}
\left( \frac{2.01}{a} \right) ^s\frac{m^{s-(t-1)}}{(s-1)!}\left( \frac{N_1}{2}\right)^t 
\nonumber \\
& = & (1+o(1))N^{.005}\sum_{s=2}^m \left( \frac{2.01}{a} \right) ^s\frac{m^s}{s!}\exp\left\{-\frac{s^2}{ 2m}\right\}m
\sum_{t=1}^{s-1} \binom{s}{t}\left( \frac{N_1}{2m}\right) ^t \nonumber \\
& \leq  & (1+o(1))\left(\frac{2m^3}{N^{.99}}\right) \sum_{s=2}^m\left(\frac{(2.01)N_1\exp\{-s/2m\}}{ 2a}\right)^s \frac{1}{s!} \nonumber \\
& = & o(1) \label{eq?????}
\end{eqnarray}
To verify that the RHS of (\ref{eq?????}) is $o(1)$ we can split the summation into 
$$S_1=\sum_{s=2}^{\lfloor m/4 \rfloor}\left(\frac{(2.01)N_1\exp\{-s/2m\}}{2a}\right)^s \frac{1}{s!}$$
and 
$$S_2=\sum_{s=\lfloor m/4 \rfloor+1}^m\left(\frac{(2.01)N_1\exp\{-s/2m\}}{2a}\right)^s \frac{1}{s!}.$$
Ignoring the term $\exp\{-s/2m\}$ we see that
\begin{eqnarray*}
S_1 & \leq & \sum_{s=2}^{\lfloor (.5025)\log N \rfloor}\frac{((1.005)\log N)^s}{s!} \\
& = & o(N^{9/10})
\end{eqnarray*}
since this latter sum is dominated by its last term.

Finally, using $\exp\{-s/2m\}<e^{-1/8}$ for $s>m/4$ we see that 
$$S_2\leq N^{(1.005)e^{-1/8}}< N^{9/10}.$$
The result follows from (\ref{eq?}) to (\ref{eq?????}).
\proofend

\end{document}